\documentclass{amsart}

\usepackage{graphicx}
\usepackage{amsmath}
\usepackage{amsfonts}
\usepackage{amssymb}

\usepackage{epic}
\usepackage{amscd}

\usepackage{amsthm}

\usepackage[all]{xy}

\usepackage[margin=1.0in]{geometry}

\usepackage{appendix}

\usepackage{multirow}
\usepackage{tikz-cd}
\usepackage{url}

\newtheorem{theorem}{Theorem}[section]

\newtheorem{conjecture}[theorem]{Conjecture}

\theoremstyle{definition}

\DeclareMathOperator{\rank}{rank}

\DeclareMathOperator{\qdim}{qdim}
\DeclareMathOperator{\qrank}{qrank}

\newcommand{\Reals}{\mathbb{R}}

\DeclareMathOperator{\Tait}{Tait}
\DeclareMathOperator{\F}{\mathbb{F}}

\begin{document}

\title{The combinatorial and gauge-theoretic foam evaluation functors
  are not the same}

\author{David Boozer} 

\begin{abstract}
Kronheimer and Mrowka used gauge theory to define a functor $J^\sharp$
from a category of webs in $\Reals^3$ to the
category of finite-dimensional vector spaces over the field of two
elements.
They also suggested a possible combinatorial replacement $J^\flat$
for $J^\sharp$, which Khovanov and Robert proved is well-defined on a
subcategory of planar webs.
We exhibit a counterexample that shows the restriction of the functor
$J^\sharp$ to the subcategory of planar webs is not the same as
$J^\flat$.
\end{abstract}

\date{\today}

\maketitle

\section{Introduction}

Kronheimer and Mrowka have outlined a new approach that could
potentially lead to the first non-computer based proof of the
four-color theorem \cite{KM-foams}.
Their approach relies on a functor $J^\sharp$, which they define using
gauge theory, from a category of webs in $\Reals^3$ and foams in
$\Reals^4$ to the category of finite-dimensional vector spaces over
the field of two elements $\F$.
A \emph{web} is an unoriented trivalent graph and a \emph{foam} is a
kind of singular cobordism between webs whose precise form is
described in \cite{KM-foams}.

The four-color theorem can be reformulated as a statement about webs.
An edge $e$ of a web is said to be a \emph{bridge} if removing $e$
increases the number of connected components of the web.
A \emph{Tait coloring} of a web is a 3-coloring of the edges of
the web such that no two edges incident on any given vertex share the
same coloring.
Given a web $K$, the \emph{Tait number $\Tait(K)$} is the number of
Tait colorings of $K$.
The four-color theorem is equivalent to the statement that every
bridgeless planar web admits at least one Tait coloring.

The functor $J^\sharp$ associates a vector space $J^\sharp(K)$ to a
web $K$ in $\Reals^3$.
An edge $e$ of a web $K$ in $\Reals^3$ is said to be an
\emph{embedded bridge} if there is a 2-sphere smoothly embedded in
$\Reals^3$ that transversely intersects $K$ in a single point that
lies on $e$.
Kronheimer and Mrowka prove the following nonvanishing theorem:

\begin{theorem}(Kronheimer--Mrowka \cite[Theorem 1.1]{KM-foams})
\label{theorem:nonvanishing}
For a web $K$ in $\Reals^3$, the vector space $J^\sharp(K)$ is zero
if and only if $K$ has an embedded bridge.
\end{theorem}

Based on some simple examples and general properties of $J^\sharp$,
they make the following conjecture, which by
Theorem \ref{theorem:nonvanishing} implies the four-color theorem:

\begin{conjecture}(Kronheimer--Mrowka \cite[Conjecture 1.2]{KM-foams})
\label{conj:tait}
For a web $K$ that lies in the plane $\Reals^2 \subset \Reals^3$, we
have $\dim J^\sharp(K) = \Tait(K)$.
\end{conjecture}

Kronheimer and Mrowka also suggested a possible combinatorial
replacement $J^\flat$ for $J^\sharp$, which they defined via a set of
rules that they conjectured would yield a well-defined functor
\cite[Section 8.3]{KM-foams}.
Khovanov and Robert later showed that $J^\flat$ is not well-defined
for arbitrary webs in $\Reals^3$ and foams in $\Reals^4$, but is well
defined provided we restrict to the subcategory of webs in
$\Reals^2$ and foams in $\Reals^3$ \cite{Khovanov}.
We will call this subcategory the category of \emph{planar} webs.
We note that planar webs are precisely those relevant to Conjecture
\ref{conj:tait}.
Based on results due to Khovanov and Robert \cite{Khovanov}, and
Kronheimer and Mrowka \cite{KM-deformation}, for any planar web $K$ we
have
\begin{align}
  \label{eqn:flat-tait-sharp}
  \dim J^\flat(K) \leq \Tait(K) \leq \dim J^\sharp(K),
\end{align}
and for a special class of \emph{reducible} planar webs (also called
\emph{simple} webs in \cite{KM-foams}), these three integers coincide:
\begin{align*}
  \dim J^\flat(K) = \Tait(K) = \dim J^\sharp(K).
\end{align*}
A proof that the restriction of $J^\sharp$ to the subcategory of
planar webs is indeed the same functor as $J^\flat$ would therefore
prove Conjecture \ref{conj:tait} and hence the four-color theorem.

It is thus of interest to understand the relationship between the
functors $J^\flat$ and $J^\sharp$.
Some insight into these functors can be gained by considering related
functors with different target categories.
In \cite{KM-deformation}, Kronheimer and Mrowka establish the second
inequality in equation (\ref{eqn:flat-tait-sharp}) by introducing a
system of local coefficients and defining a functor from the category
of webs in $\Reals^3$ to the category of modules over the ring
$\F[T_1^{\pm 1}, T_2^{\pm 1}, T_3^{\pm 1}]$.
In \cite{Khovanov}, Khovanov and Robert extend the ground field $\F$
to the graded ring $R = \F[E_1, E_2, E_3]$, where
\begin{align*}
  \deg(E_1) &= 2, &
  \deg(E_2) &= 4, &
  \deg(E_3) &= 6,
\end{align*}
and define a functor from the category of planar webs to the category
of modules over $R$.
One can obtain additional functors by base-changing to
a ring $S$ via a ring homomorphism $R \rightarrow S$.
As noted in \cite{Boozer} Corollary 3.1, the first inequality in
equation (\ref{eqn:flat-tait-sharp}) directly follows
from \cite{Khovanov} Proposition 4.18 by considering such
base-changes.

We consider here a base-change from $R$ to the graded ring $\F[E]$,
where $\deg(E) = 6$, via the ring homomorphism $R \rightarrow \F[E]$
given by
\begin{align*}
  E_1, E_2 \mapsto 0, &&
  E_3 \mapsto E.
\end{align*}
We thereby obtain a functor from the category of planar webs to the
category of modules over $\F[E]$.
We denote this functor by $\langle - \rangle$.
Given a web $K$, we say that the corresponding $\F[E]$-module
$\langle K \rangle$ is the \emph{state space} of $K$.
By \cite{Khovanov} Proposition 4.18, the state space
$\langle K \rangle$ is a free graded module of rank $\Tait(K)$.

\begin{figure}
  \centering
  \includegraphics[scale=0.5]{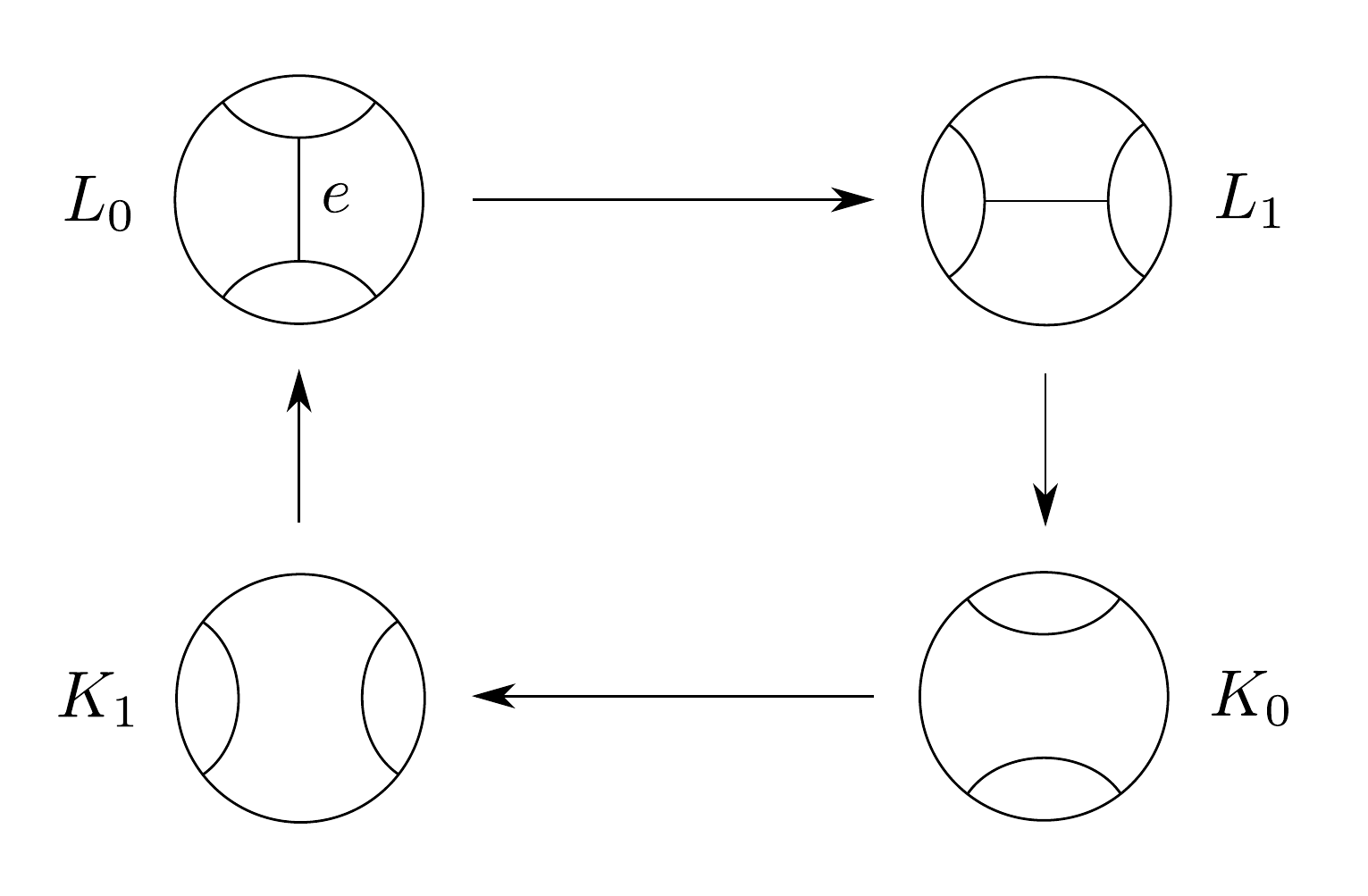}
  \caption{
    \label{fig:4-periodic}
    Given a planar web $L_0$ and choice of edge $e$, we can construct
    a square of planar webs $L_0$, $L_1$, $K_0$, and $K_1$ related by
    standard foam cobordisms in $\Reals^3$.
    The webs are identical outside of the indicated 2-balls.
  }
\end{figure}

As described in \cite{Khovanov} Section 4.3, given a planar web $L_0$
and a choice of edge $e$ we can construct a square of planar webs
$L_0$, $L_1$, $K_0$, and $K_1$ that are related by standard foam
cobordisms in $\Reals^3$ (see Figure \ref{fig:4-periodic}).
By \cite{Khovanov} Lemma 4.11, the image of this square under the
functor $\langle - \rangle$ is a 4-periodic complex:
\begin{eqnarray}
  \label{periodic-khovanov}
  \begin{tikzcd}
    \langle L_0 \rangle \arrow{r}{1} &
    \langle L_1 \rangle \arrow{d}{1} \\
    \langle K_1 \rangle \arrow{u}{1} &
    \langle K_0 \rangle. \arrow{l}[swap]{2}
  \end{tikzcd}
\end{eqnarray}
The integers indicate the degrees of the linear maps.
If we further base-change to the ground field $\F$ via the ring
homomorphism $\F[E] \rightarrow \F$, $E \mapsto 0$, we obtain a
4-periodic complex for the combinatorial functor $J^\flat$:
\begin{eqnarray}
\label{periodic-flat}
\begin{tikzcd}
  J^\flat(L_0) \arrow{r}{1} &
  J^\flat(L_1) \arrow{d}{1} \\
  J^\flat(K_1) \arrow{u}{1} &
  J^\flat(K_0). \arrow{l}[swap]{2}
\end{tikzcd}
\end{eqnarray}
In \cite{KM-exact} Lemma 10.2, Kronheimer and Mrowka describe an
analogous 4-periodic complex for the gauge-theoretic functor
$J^\sharp$:
\begin{eqnarray}
\label{periodic-km}
\begin{tikzcd}
  J^\sharp(L_0) \arrow{r} &
  J^\sharp(L_1) \arrow{d} \\
  J^\sharp(K_1) \arrow{u} &
  J^\sharp(K_0). \arrow{l}
\end{tikzcd}
\end{eqnarray}
We note that the vector spaces for $J^\flat$ are graded, but the
vector spaces for $J^\sharp$ are ungraded.
Kronheimer and Mrowka prove the following result:

\begin{theorem}(Kronheimer--Mrowka \cite[Lemma 10.3]{KM-exact})
\label{theorem:corners}
In the 4-periodic complex (\ref{periodic-km}) for $J^\sharp$, the
homology groups at diametrically opposite corners are equal.
\end{theorem}

The proof of Theorem \ref{theorem:corners} relies on the fact that the
4-periodic complex (\ref{periodic-km}) for $J^\sharp$ can be extended
to an octahedral diagram involving two additional webs that are
nonplanar.
Since $J^\flat$ is not defined for nonplanar webs, it is natural to
ask whether the analog to Theorem \ref{theorem:corners} holds for
$J^\flat$.
We answer this question in the negative by exhibiting a specific
counterexample:

\begin{theorem}
\label{theorem:homology}
For the (irreducible) web
$L_0 = W_4$ shown in Figure \ref{fig:w4} with the indicated choice of
edge $e$, the homology of the complex (\ref{periodic-flat}) for
$J^\flat$ is zero at $K_0$ but nonzero at $L_0$.
\end{theorem}

In particular, Theorems \ref{theorem:corners}
and \ref{theorem:homology} show that the restriction of the
functor $J^\sharp$ to the subcategory of planar webs is not the same
as the functor $J^\flat$.
We emphasize that this result does not refute Conjecture
\ref{conj:tait}, and thus does not invalidate Kronheimer and Mrowka's
strategy for proving the four-color theorem.

\begin{figure}
  \centering
  \includegraphics[scale=0.8]{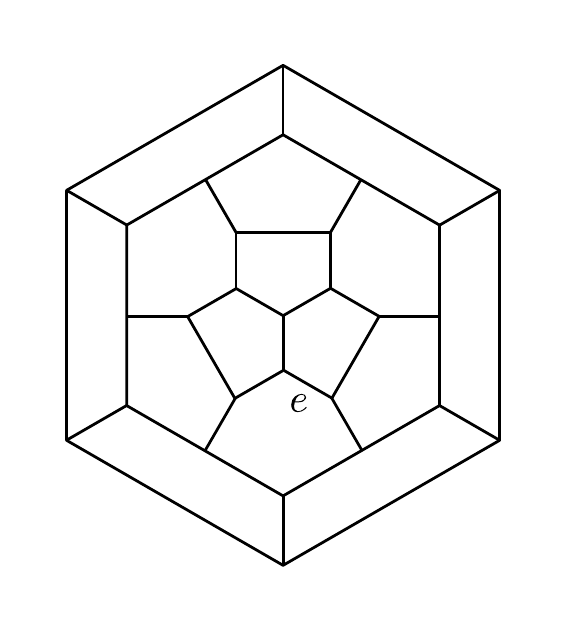}
  \caption{
    \label{fig:w4}
    Irreducible web $W_4$.
    We consider the 4-periodic complex corresponding to the indicated
    edge $e$.
  }
\end{figure}

\section{Theoretical results}

Consider the situation in which $L_0$ is irreducible but
$L_1$, $K_0$, and $K_1$ are all reducible.
Given bases for the state spaces in the complex
(\ref{periodic-khovanov}) for $\langle - \rangle$, we can express the
linear maps in the complex as $\F[E]$-valued matrices relative to
these bases.
By performing Smith decompositions of these matrices, we can decompose
each state space in the complex into a kernel, denoted by subscript
$k$, and a complement, denoted by subscript $c$:
\begin{eqnarray}
\label{periodic-decomp-L0}
\begin{tikzcd}
  \langle L_0 \rangle_k \oplus \langle L_0 \rangle _c \arrow{r}{1} &
  \langle L_1 \rangle_k \oplus \langle L_1 \rangle_c \arrow{d}{1} \\
  \langle K_1 \rangle_k \oplus \langle K_1 \rangle_c \arrow{u}{1} &
  \langle K_0 \rangle_k \oplus \langle K_0 \rangle_c. \arrow{l}[swap]{2}
\end{tikzcd}
\end{eqnarray}
We have:
\begin{theorem}
\label{theorem:torsion}
In the 4-periodic complex (\ref{periodic-khovanov}), the homology
groups are $E$-torsion.
\end{theorem}

\begin{proof}
This follows directly from \cite{Khovanov} Proposition 4.12 and
comments after the statement of Proposition 4.17, which show that the
4-periodic complex (\ref{periodic-khovanov}) becomes exact after
localizing $\F[E] \rightarrow \F[E,E^{-1}]$.
\end{proof}

By Theorem \ref{theorem:torsion}, the complement for each state space
is mapped into a submodule of full rank of the kernel for the subsequent
state space.

As described in \cite{Boozer}, we have an algorithm for constructing
bases of the state spaces for reducible webs, and we can use a
computer to construct a submodule $M$ of $\langle L_0 \rangle$ of full
rank.
The only possible difference between the submodule $M$ and the actual
state space $\langle L_0 \rangle$ is that homogeneous generators of
$M$ may be shifted upwards in degree by multiples of $\deg(E) = 6$
relative to corresponding homogeneous generators of
$\langle L_0 \rangle$.
We thus obtain a periodic complex analogous to
(\ref{periodic-decomp-L0}), but with $\langle L_0 \rangle$ replaced by
$M$:
\begin{eqnarray}
\label{periodic-decomp}
\begin{tikzcd}
  M_k \oplus M_c \arrow{r}{1} &
  \langle L_1 \rangle_k \oplus \langle L_1 \rangle_c \arrow{d}{1} \\
  \langle K_1 \rangle_k \oplus \langle K_1 \rangle_c \arrow{u}{1} &
  \langle K_0 \rangle_k \oplus \langle K_0 \rangle_c. \arrow{l}[swap]{2}
\end{tikzcd}
\end{eqnarray}
In principle, it need not be the case that the image of
$\langle K_1 \rangle \rightarrow \langle L_0 \rangle$ is contained in
$M$, but if not we can simply replace $M$ with the module spanned by
$M$ and this image.
By Theorem \ref{theorem:torsion},
the complex (\ref{periodic-decomp}) determines the ranks of the
modules in the complex (\ref{periodic-decomp-L0}).
In particular,
\begin{align*}
  &\rank(\langle L_0 \rangle_k) = \rank(M_k), &
  &\rank(\langle L_0 \rangle_c) = \rank(M_c).
\end{align*}
By computing the quantum ranks of the modules in the complex
(\ref{periodic-decomp}), we can strongly constrain the possibilities
for the complex (\ref{periodic-decomp-L0}) for $\langle - \rangle$,
which in turn strongly constrains the possibilities for the homology
of the complex (\ref{periodic-flat}) for $J^\flat$.

The relationship between the functors $\langle - \rangle$ and
$J^\flat$ is discussed in \cite[Section 3]{Boozer}.
We briefly summarize the results we will need.
Given a planar web $K$, we define a
\emph{half-foam $H$ with boundary $K$}
to be a foam cobordism in $\Reals^3$ from the empty web to $K$.
A half-foam $H$ with boundary $K$ determines elements
$\langle H \rangle \in \langle K \rangle$ and
$J^\flat(H) \in J^\flat(K)$.
If $\langle H \rangle$ is zero then $J^\flat(H)$ must be zero as well,
but in principle it may happen that $J^\flat(H)$ is zero and
$\langle H \rangle$ is nonzero, in which case we say that $H$ is a
\emph{vanishing} half-foam.
The state space $\langle K \rangle$ is freely generated by $\Tait(K)$
half-foams with boundary $K$.
Since some of the generating half-foams may be vanishing, we have
\begin{align*}
  \dim J^\flat(K) \leq \rank(\langle K \rangle) = \Tait(K), &&
  \qdim J^\flat(K) \leq \qrank(\langle K \rangle),
\end{align*}
with equality holding in the case that there are no vanishing
generators.
For $K$ reducible, there are no vanishing generators of
$\langle K \rangle$.
It is an open question as to whether there is a nonreducible web $K$
for which $\langle K \rangle$ has vanishing generators.
For every vanishing generator $\langle H_1 \rangle$ of degree $d_1$,
there must be a corresponding vanishing generator
$\langle H_2 \rangle$ of degree $d_2$ such that
$d_1 + d_2$ is a positive integer multiple of $\deg(E) = 6$.
If there are no vanishing generators of
$\langle K \rangle$, then for every generator of
$\langle K \rangle$ in degree $d$ there is a corresponding generator
in degree $-d$, so $\qrank(\langle K \rangle)$ is symmetric under
$q \rightarrow q^{-1}$.

Similar considerations apply to the linear maps corresponding to foam
cobordisms.
Consider a foam cobordism between planar webs $K$ and $L$ and the
corresponding linear maps
$\langle K \rangle \rightarrow \langle L \rangle$ and
$J^\flat(K) \rightarrow J^\flat(L)$.
Suppose $H$ is a half-foam with boundary $K$ such that
the image of $\langle H \rangle$ under
$\langle K \rangle \rightarrow \langle L \rangle$ can be expressed
as $E x$ for $x \in \langle L \rangle$.
Then $J^\flat(H)$ maps to zero under $J^\flat(K) \rightarrow J^\flat(L)$.

\section{Computer results}

We take $L_0$ to be the (irreducible) web $W_4$ shown in
Figure \ref{fig:w4} with the indicated choice of edge $e$.
The resulting webs $L_1$, $K_0$, and $K_1$ are all reducible.
We use the computer program described in \cite{Boozer} to calculate
the ranks and quantum ranks of the modules in the complex
(\ref{periodic-decomp}) and display the results in Table
\ref{table:computer-results}.
The expressions in parentheses indicate the degrees of
vanishing generators that map to zero when we base-change from $\F[E]$
to $\F$ by setting $E=0$.
In particular, there are two vanishing generators for $M$, one of
degree 1 and one of degree 5.
Since $\Tait(L_0) = \rank(M) = 180$, the fact that there are two
vanishing generators gives us a lower bound of $178$ for
$\dim J^\flat(L_0)$.
There are three possible cases for the state space
$\langle L_0 \rangle$:

\begin{enumerate}
\item
No generators are missing.
In this case $\langle L_0 \rangle = M$ and $\dim J^\flat(L_0) = 178$.
Since $\dim J^\sharp(L_0) \geq \Tait(L_0) = 180$, we already know
that this case is not consistent with $J^\flat = J^\sharp$.

\item
A generator of degree $-1$ is missing.
In this case $M$ is a proper submodule of $\langle L_0 \rangle$, with
the vanishing generator of degree 5 in $M$ shifted up relative to the
missing generator of degree $-1$ in $\langle L_0 \rangle$, and
$\dim J^\flat(L_0) = 180$.

\item
A generator of degree $-5$ is missing.
In this case $M$ is a proper submodule of $\langle L_0 \rangle$, with
the vanishing generator of degree 1 in $M$ shifted up relative to the
missing generator of degree $-5$ in $\langle L_0 \rangle$, and
$\dim J^\flat(L_0) = 180$.
\end{enumerate}

For each case, we show that the homology of the complex
(\ref{periodic-flat}) for $J^\flat$ is zero at $K_0$ but nonzero at
$L_0$, thus proving Theorem \ref{theorem:homology} from the
introduction.

\begin{table}
\begin{align*}
  \begin{array}{lll}
    \textup{module} &
    \rank &
    \qrank \\
    \\
    M &
    180 &
    q^{-6}+11q^{-4}+10q^{-3}+29q^{-2}+19q^{-1}+38+19q+29q^2+10q^3+11q^4+q^6+(q+q^5)
    \\
    M_k &
    72 &
    2q^{-4}+q^{-3}+11q^{-2}+3q^{-1}+19+3q+18q^2+3q^3+9q^4+q^6+(q+q^5)
    \\
    M_c &
    108 &
    q^{-6}+9q^{-4}+9q^{-3}+18q^{-2}+16q^{-1}+19+16q+11q^2+7q^3+2q^4
    \\
    \\
    \langle L_1 \rangle &
    168 &
    2q^{-5}+8q^{-4}+12q^{-3}+24q^{-2}+22q^{-1}+32+22q+24q^2+12q^3+8q^4+2q^5
    \\
    \langle L_1 \rangle_k &
    108 &
    q^{-5}+q^{-4}+9q^{-3}+9q^{-2}+18q^{-1}+16+19q+15q^2+11q^3+7q^4+2q^5
    \\
    \langle L_1 \rangle_c &
    60 &
    q^{-5}+7q^{-4}+3q^{-3}+15q^{-2}+4q^{-1}+16+3q+9q^2+q^3+q^4
    \\
    \\
    \langle K_0 \rangle &
    72 &
    q^{-5}+q^{-4}+10q^{-3}+3q^{-2}+19q^{-1}+4+19q+3q^2+10q^3+q^4+q^5
    \\
    \langle K_0 \rangle_k &
    60 &
    q^{-4}+7q^{-3}+3q^{-2}+15q^{-1}+4+16q+3q^2+9q^3+q^4+q^5
    \\
    \langle K_0 \rangle_c &
    12 &
    q^{-5}+3q^{-3}+4q^{-1}+3q+q^3
    \\
    \\
    \langle K_1 \rangle &
    84 &
    2q^{-5}+q^{-4}+12q^{-3}+3q^{-2}+22q^{-1}+4+22q+3q^2+12q^3+q^4+2q^5 \\
    \langle K_1 \rangle_k &
    12 &
    q^{-3}+3q^{-1}+4q+3q^3+q^5 \\
    \langle K_1 \rangle_c &
    72 &
    2q^{-5}+q^{-4}+11q^{-3}+3q^{-2}+19q^{-1}+4+18q+3q^2+9q^3+q^4+q^5
  \end{array}
\end{align*}
  \caption{
    \label{table:computer-results}
    Ranks and quantum ranks of the modules in the complex
    (\ref{periodic-decomp}).
    The expressions in parentheses indicate the degrees of
    vanishing generators that map to zero when we base-change from
    $\F[E]$ to $\F$ by setting $E=0$.
  }
\end{table}

\subsection{Case 1: no generators are missing}

From Table \ref{table:computer-results}, it follows that the
quantum ranks of the modules for $L_0$ are
\begin{align*}
    &\qrank(\langle L_0 \rangle) =
    q^{-6}+11q^{-4}+10q^{-3}+29q^{-2}+19q^{-1}+38+19q+29q^2+10q^3+11q^4+q^6+(q+q^5),
    \\
    &\qrank(\langle L_0 \rangle_k) =
    2q^{-4}+q^{-3}+11q^{-2}+3q^{-1}+19+3q+18q^2+3q^3+9q^4+q^6+(q+q^5),
    \\
    &\qrank(\langle L_0 \rangle_c) =
    q^{-6}+9q^{-4}+9q^{-3}+18q^{-2}+16q^{-1}+19+16q+11q^2+7q^3+2q^4.
\end{align*}
Thus
\begin{align*}
  &\qrank(\langle L_0 \rangle_k) -
  q \cdot \qrank(\langle K_1 \rangle_c) = 0, &
  &\qrank(\langle L_1 \rangle_k) -
  q \cdot \qrank(\langle L_0 \rangle_c) = q^{-4} - q^2, \\
  &\qrank(\langle K_1 \rangle_k) -
  q^2 \cdot \qrank(\langle K_0 \rangle_c) = 0, &
  &\qrank(\langle K_0 \rangle_k) -
  q \cdot \qrank(\langle L_1 \rangle_c) = 0.
\end{align*}
It follows that the quantum dimensions of the homology groups for the
complex (\ref{periodic-flat}) for $J^\flat$ are
\begin{align*}
  &\qdim(H(L_0)) = q, &
  &\qdim(H(L_1)) = q^{-4}, \\
  &\qdim(H(K_1)) = 1 + q^4, &
  &\qdim(H(K_0)) = 0,
\end{align*}
since we have generators of degrees $0$ and $4$ in
$\langle K_1 \rangle_c$ that map to the vanishing generators of
$\langle L_0 \rangle_k$, and we have a
generator of degree $1$ in $\langle L_0 \rangle_c$ that maps to
$E$ times a generator of degree $-4$ in $\langle L_1 \rangle_k$.
In particular, the homology is zero at $K_0$ and nonzero at $L_0$.

\subsection{Case 2: a generator of degree $-1$ is missing}

From Table \ref{table:computer-results}, it follows that the ranks and
quantum ranks of the modules for $L_0$ are
\begin{align*}
  &\qrank(\langle L_0 \rangle) =
  q^{-6}+11q^{-4}+10q^{-3}+29q^{-2}+20q^{-1}+38+20q+29q^2+10q^3+11q^4+q^6,
  \\
  &\qrank(\langle L_0 \rangle_k) =
  2q^{-4}+q^{-3}+11q^{-2}+4q^{-1}+19+4q+18q^2+3q^3+9q^4+q^6,
  \\
  &\qrank(\langle L_0 \rangle_c) =
  q^{-6}+9q^{-4}+9q^{-3}+18q^{-2}+16q^{-1}+19+16q+11q^2+7q^3+2q^4.
\end{align*}
Thus
\begin{align*}
  &\qrank(\langle L_0 \rangle_k) -
  q \cdot \qrank(\langle K_1 \rangle_c) = q^{-1} - q^5, &
  &\qrank(\langle L_1 \rangle_k) -
  q \cdot \qrank(\langle L_0 \rangle_c) = q^{-4} - q^2, \\
  &\qrank(\langle K_1 \rangle_k) -
  q^2 \cdot \qrank(\langle K_0 \rangle_c) = 0, &
  &\qrank(\langle K_0 \rangle_k) -
  q \cdot \qrank(\langle L_1 \rangle_c) = 0.
\end{align*}
It follows that the quantum dimensions of the homology groups for the
complex (\ref{periodic-flat}) for $J^\flat$ are
\begin{align*}
  &\qdim(H(L_0)) = q^{-1} + q, &
  &\qdim(H(L_1)) = q^{-4}, \\
  &\qdim(H(K_1)) = q^4, &
  &\qdim(H(K_0)) = 0,
\end{align*}
since we have a generator of degree $4$ in $\langle K_1 \rangle_c$
that maps to $E$ times a generator of degree $-1$ in
$\langle L_0 \rangle_k$ and we have a generator of degree $1$ in
$\langle L_0 \rangle_c$ that maps $E$ times a generator of degree
$-4$ in $\langle L_1 \rangle_k$.
In particular, the homology is zero at $K_0$ and nonzero at $L_0$.

\subsection{Case 3: a generator of degree $-5$ is missing}

From Table \ref{table:computer-results}, it follows that the ranks and
quantum ranks of the modules for $L_0$ are
\begin{align*}
  &\qrank(\langle L_0 \rangle) =
  q^{-6}+q^{-5}+11q^{-4}+10q^{-3}+29q^{-2}+19q^{-1}+38+
  19q+29q^2+10q^3+11q^4+q^5+q^6,
  \\
  &\qrank(\langle L_0 \rangle_k) =
  q^{-5}+2q^{-4}+q^{-3}+11q^{-2}+3q^{-1}+19+3q+18q^2+3q^3+9q^4+q^5+q^6,
  \\
  &\qrank(\langle L_0 \rangle_c) =
  q^{-6}+9q^{-4}+9q^{-3}+18q^{-2}+16q^{-1}+19+16q+11q^2+7q^3+2q^4.
\end{align*}
Thus
\begin{align*}
  &\qrank(\langle L_0 \rangle_k) -
  q \cdot \qrank(\langle K_1 \rangle_c) = q^{-5} - q, &
  &\qrank(\langle L_1 \rangle_k) -
  q \cdot \qrank(\langle L_0 \rangle_c) = q^{-4} - q^2, \\
  &\qrank(\langle K_1 \rangle_k) -
  q^2 \cdot \qrank(\langle K_0 \rangle_c) = 0, &
  &\qrank(\langle K_0 \rangle_k) -
  q \cdot \qrank(\langle L_1 \rangle_c) = 0.
\end{align*}
It follows that the quantum dimensions of the homology groups for the
complex (\ref{periodic-flat}) for $J^\flat$ are
\begin{align*}
  &\qdim(H(L_0)) = q^{-5} + q, &
  &\qdim(H(L_1)) = q^{-4}, \\
  &\qdim(H(K_1)) = 1, &
  &\qdim(H(K_0)) = 0,
\end{align*}
since we have a generator of degree $0$ in $\langle K_1 \rangle_c$
that maps to $E$ times a generator of degree $-5$ in
$\langle L_0 \rangle_k$ and we have a generator of degree $1$ in
$\langle L_0 \rangle_c$ that maps $E$ times a generator of degree
$-4$ in $\langle L_1 \rangle_k$.
In particular, the homology is zero at $K_0$ and nonzero at $L_0$.

\section*{Acknowledgments}

The author would like to thank Mikhail Khovanov for reading an
earlier version of this manuscript and providing helpful comments.

\bibliographystyle{abbrv}
\bibliography{foam-counterexample}

\end{document}